\documentclass[10pt]{article}

\usepackage{amsmath}
\usepackage{amsthm}
\usepackage{amsfonts}
\usepackage{amssymb}
\usepackage{lipsum}
\usepackage{blindtext} 
\usepackage{mathtools}
\usepackage{thmtools}
\usepackage{enumitem}
\usepackage{xcolor}
\usepackage{cite}
\usepackage[numbers]{natbib}
\usepackage{mathtools}
\mathtoolsset{showonlyrefs=true}

\usepackage{
nameref,
hyperref,
}
\usepackage{braket}
\usepackage{cancel}

\usepackage{graphicx}
\usepackage{floatrow}
\newtheorem{theorem}{Theorem}[section]
\newtheorem{corollary}{Corollary}[theorem]
\newtheorem{lemma}[theorem]{Lemma}
\newtheorem*{remark}{Remark}
\newtheorem{definition}{Definition}

\newcommand{\bit}{\begin{itemize}}
\newcommand{\eit}{\end{itemize}}
\newcommand{\bd}{\begin{description}}
\newcommand{\ed}{\end{description}}
\newcommand{\ben}{\begin{enumerate}}
\newcommand{\een}{\end{enumerate}}
\newcommand{\bqn}{\begin{equation}\begin{aligned}}
\newcommand{\eqn}{\end{aligned}\end{equation}}
\newcommand{\bgn}{\begin{gather}}
\newcommand{\egn}{\end{gather}}
\newcommand{\bqnn}{\begin{aligned}}
\newcommand{\eqnn}{\end{aligned}\end{equation}}

\newcommand{\bt}{\begin{thm}}
\newcommand{\et}{\end{thm}}
\newcommand{\bl}{\begin{lem}}
\newcommand{\el}{\end{lem}}
\newcommand{\bp}{\begin{prop}}
\newcommand{\ep}{\end{prop}}
\newcommand{\bc}{\begin{cor}}
\newcommand{\ec}{\end{cor}}
\newcommand{\bdefn}{\begin{defn}}
\newcommand{\edefn}{\end{defn}}
\newcommand{\brem}{\begin{rem}}
\newcommand{\erem}{\end{rem}}
\newcommand{\bproof}{\begin{proof}}
\newcommand{\eproof}{\end{proof}}
\newcommand{\bex}{\begin{ex}}
\newcommand{\eex}{\end{ex}}
\newcommand{\R}{\mathbb{R}}

\newcommand{\bcs}{\begin{cases}}
\newcommand{\ecs}{\end{cases}}

\newcommand{\spc}{\;\;\;}
\newcommand{\x}{x}
\newcommand{\s}{s}
\newcommand{\expp}{\text{e}}
\newcommand{\bnu}{\nu}

\newcommand{\gyp}{\gamma(t^{1/2}y^{\prime})}

\newcommand{\yp}{y^{\prime}}
\newcommand{\typ}{t^{1/2}y^{\prime}}
\newcommand{\xp}{x^{\prime}}

\newcommand{\op}{0^{\prime}}

\newcommand{\INT}{\int\limits}
\newcommand{\SUM}{\sum\limits}

\newcommand{\rr}{r}
\newcommand{\B}{B}

\newcommand{\abs}[1]{\lvert #1 \rvert}
\newcommand{\abss}[1]{\big\lvert #1 \big\rvert}

\begin{document}
\date{}
\author{E. Baspinar\thanks{University of Bologna, Department of Mathematics.} \and G. Citti\thanks{University of Bologna, Department of Mathematics.}\footnotemark[1]}
\title{A diffusion driven curvature flow}
\maketitle

\section{Introduction}

Geometrical flow is the evolution of a surface $\Sigma_0\subset\R^n$ 
whose points move with speed equal to a function of the mean curvature and direction 
normal to the evolving surface $\Sigma_t\subset\R^n$. In particular mean curvature flow has been studied by several authors such as Altschuler-Grayson \cite{altschuler1991shortening}, \cite{ecker1989mean}, Gage-Hamilton \cite{gage1986heat}, Grayson \cite{grayson1987heat} and Huisken \cite{huisken1984flow} who exploited parametric methods of differential geometry. For an overview on other geometrical flows we refer to \cite{gerhard1999geometric, giga2006surface,lu2002surface}.

However the parametric methods were insufficient if singularities occur during mean curvature flow. In order to handle such difficulties Brakke \cite{brakke2015motion} employed varifold approach via geometric measure theory and provided weak solutions to the mean curvature flow. Furthermore, Osher-Sethian \cite{osher1988fronts} used level set approach in numerical studies and later Evans-Spruck \cite{evans1991motion} and Chen-Giga-Goto \cite{chen1991uniqueness} extended level set approach in combination with the theory of viscosity solutions in order to provide generalized weak solutions to mean curvature evolution PDE. 

Here we will be interested in motion by curvature of surfaces which are expressed 
as a graph $x=(\x^{\prime},\gamma(\x^{\prime}))\in\R^n$ of a smooth function $\gamma:\R^{n-1}\rightarrow \R$. In this case the curvature 
is expressed as 
\begin{align}
K=\operatorname{div}\Bigg(\frac{\nabla\gamma}{\sqrt{1+\lvert\nabla \gamma\rvert^2}}\Bigg).
\end{align}   
A special class of curvature equations, which contains 
the motion by curvature is 
\begin{align}\label{eq:mcfq}
\partial_t\gamma = F(|\nabla\gamma|^2+1) K(\gamma), 
\end{align}
for a suitable smooth function $F$. In the special case of mean curvature flow $F(1 + |\nabla \gamma|^2)= (1 + |\nabla \gamma|^2)^{1/2}$.
\bigskip

Bence-Merriman-Osher \cite{merriman1992diffusion} provided a numerical algorithm which expresses the motion by curvature as a two step procedure with diffusion and recovery of the surface. They started with a surface $\Sigma_0$ which was the boundary of a set $C_0$ and applied heat diffusion to the 
characteristic function $\chi_{C_0}$ of the set $C_0$ for a fixed interval of time $[0, T]$:
\begin{align}
\bcs
u_t=\Delta u\spc\text{in $\R^n\times (0,\infty)$}\\
u=\chi_{C_0}\spc\text{on $\R^n\times \{t=0\}$}.
\ecs
\end{align}
At time $T>0$ they recovered a new set  as
\begin{align}\label{acca}
C_1=H(T)C_0\equiv\{\x\in\R^n :\; u(\x,T)\geq \frac{1}{2}\}.
\end{align} 
Applying iteratively this procedure, Bence-Merriman-Osher generated 
a sequence of sets $(C_j)$ and conjectured 
in \cite{merriman1992diffusion} that their 
boundaries $\Sigma_j$ evolved by mean curvature flow. The convergence of the algorithm to mean curvature flow was proven by Evans \cite{evans1993convergence}, Barles-Georglin \cite{barles1995simple}, Ishii \cite{ishii1995generalization}, Ishii-Pires-Souganidis \cite{ishii1999threshold}, Vivier \cite{vivier2000convergence}, Leoni \cite{leoni2001convergence} and Goto-Ishii-Ogawa \cite{goto2002method}. 

In his proof, Evans applied a level set approach. The sets $C_j$ were identified with sub-level sets of functions $f_j$, and expressed as $C_j=\{x\in\R^n:\;f_j(x)\geq \lambda\geq 0 \}$. Accordingly the operator $H$ in \eqref{acca} is rewritten as an operator acting on continuous functions. Then with general instruments of nonlinear semigroup theory, he proved that for every $f_0$ in the domain of the curvature operator and for 
every $t\ge 0$, the iterative application of the algorithm, converges to the 
solution of mean curvature flow. Formally, there exists the limit
\begin{align}\label{thesis}
\lim_{j\to \infty} H\bigg(\frac{t}{j}\bigg)^j f_0 , 
\end{align}
and it coincides with the motion by curvature of the level sets of the function $f_0$. 
The proof was extended to motion of graphs by curvature in Carnot groups in \cite{capogna2013sub}. 

A generalization of Bence-Merriman-Osher algorithm 
was proposed in \cite{citti2006cortical}. This version of the algorithm was 
inspired by the behavior of simple cells of the visual cortex, and 
their ability of completion of corrupted parts of images. In the original version, this alogrithm allows a surface and a density function defined on it to evolve at the same time. 
Calling $u$ the evolution in time of a measure concentrated on the surface,
they defined a  new surface as the zero 
level-set of the gradient of $u$ along normal direction. 
The authors proposed that 
$\Sigma_j(t/j)$ converges to a solution of the mean curvature flow, and 
provided a local approximation result. However a complete proof is still missing. 

\bigskip

Hence we modify the algorithm of \cite{citti2006cortical}, introducing a geometrical correction 
in the initial datum, and we provide a complete proof of its convergence. 
Motivated by the applied problem we consider evolution of surfaces which are graphs given by 
$$\Sigma_t = (x^{\prime},\gamma(x^{\prime}))\subset \R^n.$$
We first apply heat diffusion in an interval $[0, T]$ of time 
\begin{align}\label{eq:mainHeatDiffReferred}
\bcs
\partial_t u=\Delta u\spc\text{in $\R^n\times (0,\infty)$}\\
u=<\nu,r>\delta_{\Sigma_0}\spc\text{on $\Sigma_0 \times \{t=0\}$},
\ecs
\end{align}
and then define the new surface at $t=T$ as the zero level set of the gradient along some direction via
\begin{align}\label{eq:surfaceDefn2}
\Sigma_1:= \{x\in\R^n:\; \nabla_r u(x,t)=0 \},
\end{align}
where $r$ is a fixed unit vector. If $r= e_n$, the last element of the canonical basis, we 
can prove that $\Sigma_1$ is the graph of a new function $\gamma=H(t)f_0$. Then the same procedure can be iterated. 

We provide here a complete proof of the convergence of this algorithm to the solution of the curvature flow:
\begin{align}\label{eq:firstEquationAppear}
\bcs
\partial_t\gamma-\frac{<\rr, e_n>}{ <\nu_{s_0}, e_n><\nu_{s_0}, \rr>}\SUM_{i,j=1}^{n-1}\bigg(\delta_{ij}-\frac{\gamma_{x_i} \gamma_{x_j} }{1+|\nabla \gamma|^2}\bigg)\gamma_{x_ix_j}=0\quad \text{in $\Sigma_t^{\prime}\times (0,T]$}\\
\gamma=f_0\quad \text{on $\Sigma_0^{\prime}\times\{t=0\}$,}
\ecs
\end{align}
where prime denotes the first $(n-1)$ components of the corresponding terms here and everywhere it appears from now on.

Precisely, in the special case where $r=e_n$ our main result can be stated as follows:

\begin{theorem}\label{maintheorem}
If $f_0$ is a continuous and periodic function, and $\gamma$ is the unique solution to \eqref{eq:mcfq}, with $F(1+|\nabla \gamma|^2)=(1+|\nabla \gamma|^2)^{3/2} $ and initial datum $f_0$, then
\begin{align}
\gamma(x^{\prime},t)=\lim_{j\rightarrow \infty}H\Big(\frac{t}{j}\Big)f_0,\quad\text{uniformly for $x^{\prime}\in \Sigma_t^{\prime}$ and $t\geq 0$ in compact sets.}
\end{align}
\end{theorem}

The peculiarity of the algorithm proposed in \cite{citti2006cortical} relies on the fact that it models the behavior of the cells of the visual cortex, and it can be used to implement 
a surface completion algorithm. Both those properties are preserved in our algorithm (as explained in \eqref{eq:mainHeatDiffReferred} and \eqref{eq:surfaceDefn2}) through the geometric modification on the initial function. In addition, a minimization in a fixed direction is simpler to implement than the minimization in the normal direction. Furthermore the special case with $r=e_n$ is particularly important since it complies with the direction in which the graph function is defined, thus with orientation selectivity of simple cells as described by the visual cortex model. The proof of convergence of our algorithm is partially inspired by the proof of Evans \cite{evans1993convergence}, 
but we deeply simplify it, since we study in a graph setting. 

In section 2 we recall the main instruments of nonlinear semigroup theory necessary for the proof of Theorem \ref{maintheorem}, and the weak definition of viscosity solution of \eqref{eq:mcfq}. In particular we will see a general definition of the curvature operator 
$$A(\gamma) = F(1 + |\nabla \gamma|^2) K$$.
 
In section \ref{sec:propertiesOfTheOperatorH} we show that the operator $H$ is contractive.
 
In section 4 we study the local behavior of our algorithm, 
showing that each point of the surface moves 
in normal direction with a speed equal to  $A(\gamma)/(1 + |\nabla \gamma|^2)$.

Proofs of sections 3 and 4 are the novel parts of the paper, and they are more delicate than the previous ones of Evans \cite{evans1993convergence}, 
since the initial datum is given only on a surface. 
Indeed Citti-Sarti surface completion model \cite{citti2006cortical} motivates the choice of the geometrically modified density function $<\nu,r>\delta_{\Sigma_0}$ on the surface $\Sigma_0$ as the initial function. 

Finally in section 5 we conclude the proof of Theorem \ref{maintheorem}, 
applying a general theorem of Brezis and Pazy \cite{brezis1972convergence}, which is given as follows:  
\begin{theorem}\label{BrezisPazyteo}
Suppose that there exists a family of contractive operators 
$\{H(t)\}_{t\ge 0}$ satisfying
\begin{align}\label{hyp}
(I+\lambda A)^{-1} f= \lim_{t\to 0^+} \bigg( I+\lambda t^{-1} (I-H(t)) \bigg)^{-1} f,
\end{align}
for every $f,g\in B$ and $\lambda>0$. Then for every $f\in \bar D(A)$ and $t\ge 0$,
one has: 
\begin{align}\label{thesis}
M(t)f=\lim_{j\to \infty} H\bigg(\frac{t}{j}\bigg)^j f , \text{ uniformly for } t \text{ in compact sets,}\end{align}
where $M(t)f$ is the semigroup generated by $A$ (see the definition in Theorem \ref{thm:FunGenThm} below).
\end{theorem}

\section{Nonlinear semigroups and curvature flow of graphs}

\begin{definition}\label{diss} Let $\B$ be a Banach space, and A a $\B-$valued nonlinear operator with domain $D(A)\subset \B$. We say that $-A$ is  $m-$dissipative if
\begin{itemize}\item{$R(I+\lambda A)=\B$ for every $\lambda>0$,}
\item{its resolvent $J_\lambda=(I+\lambda A)^{-1}$ is
a single-valued contraction.}
\end{itemize}
\end{definition}

\bigskip

For non dissipative operators $A$ it is possible to apply 
the fundamental generation theorem of Crandall and Liggett \cite[Theorem I, p.266]{crandall1971generation} which gives a weak definition of solution to evolution equations in the setting of a nonlinear semigroup:
\begin{theorem}\label{thm:FunGenThm}
If $A$ is non dissipative operator on a Banach space $B$, then for all $f\in B$ the limit
\begin{align}
M(t)(f):=\lim_{j\rightarrow \infty, \lambda j\rightarrow t}(I+\lambda A)^{-j}f,
\end{align}
exists locally uniformly in $t$. This limit is called nonlinear semigroup solution 
generated by $A$. 
\end{theorem}

From now on $\B$ will denote the space of periodic $\alpha$-H\"{o}lder continuous functions 
and  $(\B,|| \cdot ||)$ will be the Banach space obtained by endowing $\B$ with the sup norm $||\cdot||$. Then Schauder theory \cite{schauder1934lineare,schauder1934numerische} and Sobolev embedding theorems \cite{sobolev1963theorem,sobolev2008some} guarantee for a choice of $f\in L^p$ the existence and uniqueness of $W^{2,p}$ solutions to \eqref{eq:firstEquationAppear} as long as $<r,e_n>\neq 0$.

\begin{definition}\label{defn:viscositySolutions}
A continuous function $\gamma:\R^{n-1}\to\R$ is a weak sub-solution (resp. a super-solution) of
\begin{align}\label{intro-generator}
\gamma(x^{\prime})-\lambda \sum_{i,j=1}^{n-1}\bigg(\delta_{ij}-\frac{ \gamma_{x_i}(x^{\prime})\gamma_{x_j}(x^{\prime})}{1+|\nabla \gamma(x^{\prime})|^2}\bigg) \gamma_{x_ix_j}(x^{\prime})=f(x^{\prime}),
\end{align}
in $\R$ if for every $x^{\prime}\in \R^{n-1}$ and smooth $\phi:\R^{n-1}\to\R$ such that $\gamma-\phi$ has a maximum (resp. a minimum) at $x^{\prime}$
one must have $$\gamma(x^{\prime})-\lambda\sum_{i,j=1}^{n-1}\bigg(\delta_{ij}-\frac{\phi_{x_i}(x^{\prime})\phi_{x_j}(x^{\prime})}{1+|\nabla \phi(x^{\prime})|}\bigg) \phi_{x_ix_j}(x^{\prime})\le (resp. \ \ge ) \ f(x^{\prime}).$$ Solutions are  functions which are  simultaneously  super-solutions and sub-solutions.
\end{definition}

\begin{definition}\label{def-gen}
We say that $\gamma\in \B$ belongs to the domain of $A$ if there exists $f\in \B$ and $\lambda>0$ such that
$\gamma$ is a weak solution (in the sense of Theorem \ref{thm:FunGenThm}) of
\begin{align}\label{generator}
\gamma-\lambda \sum_{i,j=1}^{n-1} a_{ij}(\nabla \gamma)\gamma_{x_ix_j}=f, \text{ where }a_{ij}(\nabla \gamma):=\bigg(\delta_{ij}-\frac{ \gamma_{x_i}\gamma_{x_j}}{1+|\nabla \gamma|^2}\bigg)
\end{align}
in $\R^{n-1}$. In this case we will write 
\begin{equation}\label{eq:equationOPeratorial}
(I+\lambda A) \gamma =f.
\end{equation}
\end{definition}

\bigskip

Clearly $A$ is defined in a dense set of $\B$, and as being argued in \cite[Theorem 2.3, Theorem 2.5]{evans1993convergence}), it is possible to show that
$A$ is non dissipative and that its generated weak semigroup solution coincides with the viscosity solution to \eqref{eq:firstEquationAppear} (see \cite{biton2001nonlinear, evans1987nonlinear,crandall1984some,evans1993convergence,juutinen2001equivalence,ishii1990viscosity} for details).

\section{Properties of operator $H$}
\label{sec:propertiesOfTheOperatorH}
In this section we formally define the operator $H$ mentioned in the introduction, and we prove that it is contractive. 

We assume that the initial surface $\Sigma_0$ is the graph of some smooth function $\gamma$, and evolve the measure $u_0=<\nu,r>\delta_{\Sigma_0}$ for an interval of time $[0,T]$. Then the solution of \eqref{eq:firstEquationAppear} 
can be written as 
\begin{align}\label{eq:fundSolnExplicitExpression}
u(s,t)=\frac{1}{(4\pi t)^{^{n/2}}}\int\limits_{\Sigma_t}\sqrt{4\pi t}\;\expp^{-\lvert\x-s\rvert^2/4t}u_0(\x)d\sigma_{t},
\end{align}
where $d\sigma_t$ denotes the surface measure element on $\Sigma_t$.

If $f_0\in \B$, we send its graph $\Sigma_0$ to a new set defined via $\Sigma_t = \Big(s^{\prime},\big(H(t)f_0\big)(s^{\prime})\Big)= (s^{\prime},\gamma^t(s^{\prime}))$ with smooth function $\gamma^t$. In order to prove that $H$ induces a flow on $\B$, we need to show that $ \Sigma_t$ is the graph of a periodic, $\alpha$-H\"{o}lder continuous function $\gamma^t$:
\begin{lemma}\label{defflow}Let $u$ be the function defined as in \eqref{eq:fundSolnExplicitExpression} and $f_0\in \B$. Then for every $s^{\prime}\in \R^{n-1}$ and for every $t>0$ there exists a unique value $\gamma^t(s^{\prime})$ such that 
$$ u_{x_n}(s^{\prime}, \gamma^t(s'), t)=0\quad\text{and}\quad  u_{{x_n}{x_n}}(s^{\prime}, \gamma^t(s^{\prime}), t)<0.$$
As a consequence, for any $t>0$ there exists a function  $\gamma^t\in \B$ such that
$$\Sigma_t=(s^{\prime},(H(t)f_0)(s^{\prime}))= (s^{\prime},\gamma^t(s^{\prime})).$$
\end{lemma}
\begin{proof}
We want to compute the maximum of $u$ in the vertical direction $e_n$ via
\begin{align}
 u_{x_n}(\s,t)=0,
\end{align}
and show that
\begin{align}
u_{{x_n}{x_n}}(\s,t)<0,
\end{align}
where the gradient vanishes. Arguing as in Theorem \ref{thm:thm1} we find: 
\begin{align}
u_{{x_n}{x_n}} (\s,t)=-\frac{1}{(4\pi t)^{^{(n-1)/2}}}\int\limits_{\Sigma_t}\Big( 1 -\frac{(\gamma(x')-tv_{\s_0}\bnu_{n})^2}{2t}\Big)\expp^{-\lvert\x-tv_{\s_0}\bnu_{\s_0}\rvert^2/4t}u_0(\x)d\sigma_{\x}\\
=-\int\limits_{\R^{n-1}} \Big( 1 -\frac{\big(\gamma(t^{1/2}(U^{-1}z)')\big)^2}{2t}\Big)\expp^{-|z|^2/4}\abs{U^{-1}}\,dz'+O(\expp^{t^{3/2}})\\
=-\int\limits_{\R^{n-1}} \Big( 1 -\frac{\SUM_{i,j=1}^{n-1}\gamma_{x_i}(0^{\prime}) U^{-1}_{ih}z_h\gamma_{x_j}(0^{\prime}) U^{-1}_{jk}z_k}{2}\Big)\expp^{-|z|^2/4}\abs{U^{-1}}\,dz'+O(\expp^{t^{3/2}}).
\end{align} 
As in Theorem \ref{thm:thm1} we consider $h=k$.
Then we obtain: 
\begin{align}
\begin{split}
u_{{x_n}{x_n}}(\s,t)
=&-\int\limits_{\R^{n-1}} \expp^{-|z|^2/4}\abs{U^{-1}}\,dz' + \frac{\SUM_{i,j=1}^{n-1}\gamma_{x_i}(0^{\prime}) U^{-1}_{ih}\gamma_{x_j}(0^{\prime}) U^{-1}_{jh}}{2}\int\limits_{\R^{n-1}}z_h^2\expp^{-|z|^2/4}\abs{U^{-1}}\,dz'+O(\expp^{t^{3/2}}).
\end{split}
\end{align} 

Note that $$
\frac{1}{2}\int\limits_{\R^{n-1}}z_h^2\expp^{-|z|^2/4}dz'= \int\limits_{\R^{n-1}}\expp^{-|z|^2/4}dz',$$
and also $\SUM_{h=1}^{n-1} U^{-1}_{ih}U^{-1}_{jh} = g^{ij}$, resulting in
$$\SUM_{i,j=1}^{n-1}\gamma_{x_i}(0^{\prime}) U^{-1}_{ih}\gamma_{x_i}(0^{\prime}) U^{-1}_{jh}= \SUM_{i,j=1}^{n-1} \Big(\delta_{ij} 
- \frac{\gamma_{x_i}(0^{\prime}) \gamma_{x_j}(0^{\prime})}{1 + |\nabla\gamma(0^{\prime})|^2}\Big)\gamma_{x_i}(0^{\prime}) \gamma_{x_j}(0^{\prime})= \frac{|\nabla\gamma(0^{\prime})|^2}{1 + |\nabla\gamma(0^{\prime})|^2}.$$

We conclude that 
\begin{align}
u_{{x_n}{x_n}}(\s,t)
=&-\int\limits_{\R^{n-1}} \frac{1}{1 + |\nabla\gamma(0^{\prime})|^2}\expp^{-|z|^2/4}\abs{U^{-1}}\,dz'<0.
\end{align} 
\end{proof}

\begin{remark} For every fixed $s$ and $t>0$ we have proven that 
$u_{x_n}(s', \cdot , t)$ vanishes only at the point  $\gamma^t(s')$. 
Then note that
\begin{equation}
\text{$u_{x_n}(s', q, t)<0$ for $q>\gamma^t(s')$  and  $u_{x_n}(s', q, t)>0$ for $q<\gamma^t(s')$.}
\end{equation}
\end{remark}

\begin{theorem} \label{Hlip}  For each $t \geq 0$ the flow  $H(t): \B \to \B$ just defined has the following properties
 \begin{enumerate}
\item[(1)] If $C$ is a real constant, then $H(t)(\gamma+C) = H(t)\gamma +C$,
\item[(2)] If $\gamma  \leq \mu$ then $\gamma^t=H(t)\gamma  \leq H(t)\mu=\mu^t$,
\item[(3)] $H(t)$ is a contraction on $\B$, i.e.,
$$
||H(t)\gamma - H(t)\mu || \leq || \gamma-\mu||.
$$
 \end{enumerate}
\proof  
Assertion (1) follows from the definition. For assertions (2) and (3), consider the evolution problem given by
\begin{align}
\bcs
\partial_t=\Delta u\quad\text{in $\R^n\times (0,\infty)$}\\
u=<\nu, r> \delta_{\Sigma_0}\quad \text{at time $t=0$}.
\ecs 
\end{align} 
Let us call $u_\gamma$ and $u_\mu$ the solutions with initial datum defined by the graphs of $\gamma$ and $\mu$, respectively. Hence 
\begin{align}
u_\gamma(s,t)=&\frac{1}{(4\pi t)^{n/2}}\INT_{\Sigma_t^{\prime}} (x'-s', \gamma(x')-s_n) \expp^{-\frac{ |x'- s', \gamma(x')-s_n|^2}{4t}}dx',\\\
\end{align} 
and $u_\mu(s)$ has a similar expression in terms of $\mu$. 
Note that the function 
$$(\gamma-s_n) \expp^{-\frac{ |\gamma-s_n|^2}{t}}$$
decreases as a function of $\gamma$ for $t>0$ small. It follows that if $\gamma \leq \mu$ and $\gamma^t(s')$is the function defined in Lemma \ref{defflow}, then 
$$0 = u_\gamma(s', \gamma^t(s')) \geq u_\mu(s', \gamma^t(s')).$$
This implies that $\gamma^t=H(t)\gamma \leq H(t)\mu= \mu^t$ from which assertion (2) follows. We remark that comparison principle for intrinsic functions $\mu$ and $\gamma$ becomes valid as a direct consequence of assertion (2).

Assertion (3) follows from assertions (1) and (2). Let us choose $s^{\prime}$ such that
$$H(t)\gamma(s^{\prime}) - H(t)\mu(s^{\prime}) > ||H(t)\gamma(s^{\prime}) - H(t)\mu(s^{\prime})|| - \epsilon,$$
for each $\epsilon>0$ and call
$$H(t)\gamma(s^{\prime})=\alpha,\quad H(t)\mu(s^{\prime}) =\beta.$$
By assertion (2) we have
$$H(t)\gamma(s^{\prime}) - H(t)({\mu- \beta + \alpha - \epsilon})(s^{\prime})>0,$$
and it implies by assertion (1) that there exists a point $y$ such that
\begin{align}\label{eq:fromwhichAssertionThree}
\gamma(y^{\prime}) - (\mu(y^{\prime}) -  \beta + \alpha - \epsilon)>0.
\end{align}
Finally \eqref{eq:fromwhichAssertionThree} results in, by definitions of $\alpha$ and $\beta$, that  
$$ ||H(t)\gamma - H(t)\mu|| < H(t)\gamma(s^{\prime}) - H(t)\mu(s^{\prime}) + \epsilon < \gamma(y^{\prime}) - \mu(y^{\prime}) + 2\epsilon \leq ||\gamma - \mu|| + 2\epsilon, $$
from which assertion (3) follows.
\endproof
\end{theorem}

\section{Local properties of the evolution}
\label{sec:secFixedDirection}

In this section we prove that under the action of the proposed algorithm, each point of initial surface $\Sigma_0$ moves in the normal direction with speed equal to \begin{align}
\frac{<\rr, e_n>}{ <\nu_{s_0}, e_n><\nu_{s_0}, \rr>} K+O(t^{1/2}).
\end{align}

For all $t > 0 $ denote $u(x, t)$  the solution of the Cauchy problem
$$
\bcs
\partial_t u=\Delta u\spc\text{in $\R^n\times (0,\infty)$}\\
u(.,0)=u_0(.)=<\nu,r>\delta_{\Sigma_0}(.)\quad\text{on $\Sigma_0 \times \{t=0\}$},
\ecs
$$

With use of the same notations from the previous section, 
we denote the unit normal to the surface with $\nu_{s_0}$ at $s_0\in\Sigma_0$ and 
select $v_{s_0}$ such that 
\begin{align}
s = s_0 + t v_{s_0}\nu_{s_0} \in \Sigma_t.
\end{align}
Then the following result holds:
\begin{theorem}\label{thm:thm1}
\begin{align}
v_{\s_0}= \frac{<\rr, e_n>}{ <\nu_{s_0}, e_n><\nu_{s_0}, \rr>} K+O(t^{1/2})\spc \text{as $t\rightarrow 0$},
\end{align}
where $K$ is the mean curvature (computed with respect to $\bnu_{\s_0}$) at $\s_0\in \Sigma_0$.
\end{theorem}
\begin{proof}
We may assume $\s_0=0=(0,\dots,0)\in\Sigma_0\subset \R^n$ without losing the generality. 
Then at fixed time $t>0$ we have
\begin{align}\label{eq:surfaceDefnModifiedFixed}
\nabla_{r} u(\s,t)=0.
\end{align}
More precisely,
\begin{align}
0=&<\nabla u(\s,t),r>=\frac{1}{(4\pi t)^{^{(n-1)/2}}}\int\limits_{\Sigma_0}<\x-tv_{\s_0}\bnu_{\s_0},r>\expp^{-\lvert\x-tv_{\s_0}\bnu_{\s_0}\rvert^2/4t}u_0(\x)d\sigma_t.
\end{align}

Consider $\Sigma_0$ as the graph $(x^{\prime},\gamma(x^{\prime}))$ of smooth function $\gamma:\R^{n-1}\rightarrow \R$ and write
\begin{align}
0=&\INT_{\Sigma_t^{\prime}} <\big((x', \gamma(x'))-tv_{\s_0}\bnu_{\s_0}\big),r>\expp^{-\frac{\lvert((x', \gamma(x'))-tv_{\s_0}\bnu_{\s_0}\rvert^2}{4t}}u_0((x', \gamma(x')))\sqrt{1 + |\nabla \gamma(x')|^2}\,dx'.
\end{align} 

Now substitute $\yp=t^{-1/2}\xp$ and note that 
$$u_0\Big(\typ,\gyp\Big)=u_0(0)+O(t^{1/2}\lvert\yp\rvert),
$$
$$
|\nabla \gamma(\typ)|^2= |\nabla \gamma(0^{\prime})|^2+O(t^{1/2}\lvert\yp\rvert).
$$
Then we obtain
\begin{align}\label{eq:integralInnn}
0=&\INT_{t^{-1/2}\Sigma_t^{\prime}}<\big((t^{1/2}y', \gamma(t^{1/2}y'))-tv_{\s_0}\bnu_{\s_0}\big),r>\expp^{-\frac{\lvert((t^{1/2}y', \gamma(t^{1/2}y'))-tv_{\s_0}\bnu_{\s_0}\rvert^2}{4t}}d\yp+O(\expp^{-\alpha/2t}).
\end{align} 

Taylor development gives
\begin{align}
&\gyp=\sum\limits_{i=1}^{n-1}\gamma_{x_i}(\op)t^{1/2}y_i+O(t^{1/2}\lvert\yp \rvert^3),
\end{align}
hence the argument of the exponential in \eqref{eq:integralInnn} becomes
$$\frac{\lvert((t^{1/2}y', \gamma(t^{1/2}y'))-tv_{\s_0}\bnu_{\s_0}\rvert^2}{4t}
=\frac{\lvert (t^{1/2}y', \gamma(t^{1/2}y'))\rvert^2}{4t} + O(t)$$
$$=\frac{|y'|^2 +  |\gamma(t^{1/2}y'))|^2}{t}+ O(t)=(\delta_{ij} + \gamma_{x_i}(0^{\prime})\gamma_{x_j}(0^{\prime}))y_iy_j + O(t).$$

Since the matrix consisting of $$g_{ij}= \delta_{ij} + \gamma_{x_i}(0^{\prime})\gamma_{x_j}(0^{\prime}),$$ 
is positive definite there exists a matrix $U$ such that $g_{ij} = (U^T)_{ih}U_{hj}$, where the left and right sub-indices denote row and column positions, respectively, in corresponding matrices. 
Then with the change of variable  $z=Uy$ we can write \eqref{eq:integralInnn} as
\begin{align}
0&=\int\limits_{\R^{n-1}}<\big((t^{1/2}(U^{-1}z)', \gamma(t^{1/2}(U^{-1}z)'))-tv_{\s_0}\bnu_{\s_0}\big), r>\expp^{-|z|^2/4}\abs{U^{-1}}\,dz^{\prime}+O(\expp^{-\alpha/2t})\\
&=\int\limits_{\R^{n-1}}\big(\gamma(t^{1/2}(U^{-1}z)')r_n -tv_{\s_0}<\bnu_{\s_0},r>\big)\expp^{-|z|^2/4}\abs{U^{-1}}\,dz^{\prime}+O(\expp^{-\alpha/2t}),
\end{align}
as $t\rightarrow 0$. Further expanding function $\gamma$, we obtain
\begin{align}\label{eq:intToBeComputed}
0&=\int\limits_{\R^{n-1}}\Big(r_n\sum\limits_{i=1}^{n-1}\gamma_{x_i}(\op)t^{1/2}U^{-1}_{ih}z_h+\frac{r_n}{2}\sum\limits_{i,j=1}^{n-1}\gamma_{x_ix_j}(\op)tU^{-1}_{ih}z_hU^{-1}_{jk}z_k-tv_{\s_0}<\bnu_{\s_0},r>\Big)\\
&\expp^{-|z|^2/4}\abs{U^{-1}}\,dz^{\prime}+O(t^{3/2}).
\end{align}

The first order term in \eqref{eq:intToBeComputed} vanishes due to the Euclidean symmetry. 
The second order term with $h\not=k$ also vanishes. Hence we are left with
\begin{align}
0&=\int\limits_{\R^{n-1}}\Big(\frac{r_n}{2}\sum\limits_{i,j=1}^{n-1}\gamma_{x_ix_j}(\op)tU^{-1}_{ih}U^{-1}_{jh} |z_h|^2-tv_{\s_0}<\bnu_{\s_0},r>\Big)\expp^{-|z|^2/4}\abs{U^{-1}}\,dz^{\prime}+O(t^{3/2})=\\
&=\int\limits_{\R^{n-1}}\Big(r_n\sum\limits_{i,j=1}^{n-1}\gamma_{x_ix_j}(\op)tU^{-1}_{ih}U^{-1}_{jh} -tv_{\s_0}<\bnu_{\s_0},r>\Big)\expp^{-|z|^2/4}\abs{U^{-1}}\;dz^{\prime}+O(t^{3/2}).
\end{align}

Since $$\sum\limits_{i,j=1}^{n-1}\gamma_{x_ix_j}(\op)U^{-1}_{ih}U^{-1}_{jh} =\SUM_{i,j=1}^{n-1}\Big(
\delta_{ij} - \frac{\gamma_{x_i}(\op) \gamma_{x_j}(\op)}{1 + |\nabla \gamma(\op)|^2}\Big)\gamma_{x_ix_j},$$
the integral in \eqref{eq:intToBeComputed} boils down to
\begin{align}
0=r_n \Big(\Delta\gamma(\op)-\frac{\sum\limits_{i,j=1}^{n-1}\gamma_{x_i}(\op)\gamma_{x_j}(\op)\gamma_{x_ix_j}(\op)}{1+\lvert\nabla\gamma(\op)\rvert^2}\Big)-v_{\s_0}<\bnu_{\s_0},r>+O(t^{1/2}),
\end{align}
which implies
\begin{align}\label{eq:resultAlgorithmSc}
v_{\s_0}=\frac{<r,e_n>(1+\lvert\nabla\gamma(\op)\rvert^2)^{1/2}}{<r,\bnu_{\s_0}>}\;\text{div}\Bigg(\frac{\nabla\gamma(\op)}{\sqrt{1+\lvert\nabla\gamma(\op)\rvert^2}} \Bigg)+O(t^{1/2}),
\end{align}
as $t\rightarrow 0$.
\end{proof}

\begin{corollary}\label{cor:corollary1} 
If $r=e_n$ then 
\begin{align}\label{cccc}
v_{\s_0}=(1+\abs{\nabla\gamma(s_0^{\prime})}^2)\;\text{div}\Bigg(\frac{\nabla\gamma(s_0^{\prime})}{\sqrt{1+\lvert\nabla\gamma(s_0^{\prime})\rvert^2}} \Bigg)+O(t^{1/2}),\quad \text{as $t\rightarrow 0$,}
\end{align}
for all $s_0\in \Sigma_0$.
\end{corollary}
\begin{proof}
Straightforward computation from Theorem \ref{thm:thm1}.
\end{proof}

\begin{corollary} 
If $r=\nu_{s_0}$ then 
\begin{align}\label{eq:corollaryCoinciding2}
v_{\s_0}=\text{div}\Bigg(\frac{\nabla\gamma(s_0^{\prime})}{\sqrt{1+\lvert\nabla\gamma(s_0^{\prime})\rvert^2}} \Bigg)+O(t^{1/2}),\quad \text{as $t\rightarrow 0$,}
\end{align}
for all $s_0\in\Sigma_0$.
\end{corollary}
\begin{proof}
Straightforward computation from Theorem \ref{thm:thm1}.
\end{proof}

\section{Main result}
\label{sec:mainResult}
We provide here the proof of Theorem \ref{maintheorem}, which follows in the same way as in the proof of \cite[Theorem 2]{capogna2013sub} with the only change of the corresponding evolution equation. 
\begin{proof}[Proof of Theorem \ref{maintheorem}]
We only have to show that \eqref{hyp} holds for $\lambda=1$. For this purpose, by following \cite[Theorem 5.1]{evans1993convergence} and \cite[Theorem 2]{capogna2013sub}, we define for $t>0$ and $f\in B$ that
\begin{align}
\gamma_t:=\big(I+t^{-1}(I-H(t))\big)^{-1}f\quad\text{and}\quad A_t \gamma:=\frac{\gamma-H(t)\gamma}{t}.
\end{align}
Due to \cite[Theorem 2.3]{evans1993convergence} $-A$ is non dissipative and thus so $-A_t$ is, implying for all $x,y\in\Sigma_t$ and $t> 0$ that
\begin{align}\label{eq:towardsBoundednessEquicont}
\sup_{x\in \Sigma_t}\abss{\gamma_{t}(yx)-\gamma_t(x)}\leq \sup_{x\in \Sigma_t}\abss{f(yx)-f(x)}.
\end{align}
Note that \eqref{eq:towardsBoundednessEquicont} results in a bounded and equicontinuous family $\{\gamma_t \}_{t\in (0,1]}$. Therefore Arzel\`{a}-Ascoli theorem \cite{arzela1895sulle, ascoli1884curve} is valid.

Let $\phi\in C^{\infty}(\Sigma_t^{\prime})$ such that $\gamma_t-\phi$ has a positive maximum at $x_0^{\prime}$. We can always assume that the maximum is strict, adding a suitable power of the gauge distance if it is needed, as for example in \cite{capogna2009generalized}. Since $\gamma_{t_k}\rightarrow \gamma$ uniformly on compact sets then one can find a sequence of points $x_k^{\prime}\rightarrow x_0^{\prime}$ as $k\rightarrow \infty$ such that $\gamma_{t_k}-\phi$ has a positive maximum at $x_k^{\prime}$ and 
\begin{align}
(H(t_k)\gamma_{t_k})(x_k^{\prime})-(H(t_k)\phi)(x_k^{\prime})\leq \gamma_{t_k}(x_k^{\prime})-\phi(x_k^{\prime}),\spc\text{i.e.,}\spc A_{t_k}\phi(x_k^{\prime})\leq A_{t_k}\gamma_{t_k}(x_k^{\prime}).
\end{align}
Since $\gamma_t+A_t\gamma_t=f$ then
\begin{align}\label{eq:towardsTheEnddd}
\gamma_{t_k}(x_k^{\prime})+\frac{\phi(x_k^{\prime})-(H(t_k)\phi)(x_k^{\prime})}{t_k}\leq f(x_k^{\prime}).
\end{align}
\end{proof}
We find from \eqref{eq:towardsTheEnddd} via Corollary \ref{cor:corollary1} (if we replace $f$ with $\phi$) that
\begin{align}\displaystyle
\gamma_{t_k}(x_k^{\prime})-\sqrt{1+\abs{\nabla\gamma(x_k^{\prime})}^2}\SUM_{i,j=1}^{n-1}\Big(\delta_{ij}-\frac{\phi_{x_i}(x_k^{\prime})\phi_{x_j}(x_k^{\prime})}{1+\abs{\nabla \phi(x_k^{\prime})}^2}\Big)\phi_{x_ix_j}(x_k^{\prime})+o(1)\leq f(x_k^{\prime}).
\end{align} 
Letting $k\rightarrow \infty$ we establish that $\gamma$ is a weak sub-solution of \eqref{eq:equationOPeratorial} with $\lambda=1$. Through the same reasoning one can prove that $\gamma$ is also a sup-solution and it completes the proof.

\bibliographystyle{chicagoa}
\bibliography{baspinar2016_A_diffusion}

\end{document}